\documentclass[10pt,a4paper,twoside]{article}

\usepackage{amssymb,amscd,amsfonts,amsbsy}
\usepackage{enumerate}
\usepackage{amsmath,amsthm}
\usepackage{mathrsfs}
\usepackage{epsf,epsfig}
\usepackage{pdfsync}
\usepackage[all]{xy}

\numberwithin{equation}{section}

\newtheorem{proposition}{Proposition}[section]
\newtheorem{theorem}[proposition]{Theorem}

\newtheorem{corollary}[proposition]{Corollary}

\theoremstyle{definition}

\usepackage{color}

\definecolor{gr}{rgb}   {0., 0.8, 0. } 
\definecolor{bl}{rgb}   {0., 0.5, 1. } 
\definecolor{mg}{rgb}   {0.7, 0., 0.7} 

\begin{document}

\title{A three lines proof for traces of $H^1$ functions\\ on special Lipschitz domains}

\author{Sylvie Monniaux}

\date{Aix-Marseille Universit\'e, CNRS, Centrale Marseille, I2M, UMR 7373\\
13453 Marseille, France}

\maketitle


\section{Introduction}

It is well known (see \cite[Theorem~1.2]{Necas}) that for a bounded Lipschitz domain 
$\Omega\subset{\mathbb{R}}^n$, the trace operator 
${\rm Tr}_{|_{\partial\Omega}}\,:{\mathscr{C}}(\overline{\Omega})\to{\mathscr{C}}(\partial\Omega)$ 
restricted to ${\mathscr{C}}(\overline{\Omega})\cap H^1(\Omega)$
extends to a bounded operator from $H^1(\Omega)$ to $L^2(\partial\Omega)$ and the following
estimate holds:
\begin{equation}
\label{eq1}
\|{\rm Tr}_{|_{\partial\Omega}}\,u\|_{L^2(\partial\Omega)}\le 
C\,\bigl(\|u\|_{L^2(\Omega)}+\|\nabla u\|_{L^2(\Omega,{\mathbb{R}}^n)}\bigr)
\qquad \mbox{for all }u\in H^1(\Omega),
\end{equation}
where $C=C(\Omega)>0$ is a constant depending on the domain $\Omega$. 
This result can be proved via a simple integration by parts and Cauchy-Schwarz inequality 
if the domain is the upper graph of a Lipschitz function, i.e.,
\begin{equation}
\label{omega}
\Omega=\bigl\{x=(x_h,x_n)\in{\mathbb{R}}^{n-1}\times{\mathbb{R}} ; x_n>\omega(x_h)\bigr\}
\end{equation}
where $\omega:{\mathbb{R}}^{n-1}\to{\mathbb{R}}$ is a globally Lipschitz function.

\section{The result}

Let $\Omega\subset{\mathbb{R}}^n$ be a domain of the form \eqref{omega}. The exterior unit 
normal $\nu$ of $\Omega$ at a point $x=(x_h,\omega(x_h))$ on the boundary $\Gamma$ of 
$\Omega$:
$$
\Gamma:=\bigl\{x=(x_h,x_n)\in{\mathbb{R}}^{n-1}\times{\mathbb{R}} ; x_n=\omega(x_h)\bigr\}
$$
is given by
$$
\nu(x_h,\omega(x_h))=\frac{1}{\sqrt{1+|\nabla_h\omega(x_h)|^2}} \,(\nabla_h\omega(x_h),-1)
$$
($\nabla_h$ denotes the ``horizontal gradient'' on ${\mathbb{R}}^{n-1}$ acting on the ``horizontal 
variable'' $x_h$). We denote by $\theta\in[0,\frac{\pi}{2})$ the angle 
\begin{equation}
\label{theta}
\theta=\arccos\,\Bigl(\inf_{x_h\in{\mathbb{R}}^{n-1}}\frac{1}{\sqrt{1+|\nabla_h\omega(x_h)|^2}}\Bigr),
\end{equation}
so that in particular for $e=(0_{{\mathbb{R}}^{n-1}},1)$ the ``vertical'' direction, we have
\begin{equation}
\label{costheta}
-e\cdot \nu(x_h,\omega(x_h))=\frac{1}{\sqrt{1+|\nabla_h\omega(x_h)|^2}}\ge\cos\theta>0,\quad
\mbox{for all }x_h\in{\mathbb{R}}^{n-1}.
\end{equation}

\begin{theorem}
\label{H1}
Let $\Omega\subset{\mathbb{R}}^n$ be as above. 
Let $\varphi:{\mathbb{R}}^n\to{\mathbb{R}}$ be a smooth function with compact support. 
Then 
\begin{equation}
\label{traceH1}
\int_\Gamma|\varphi|^2\,{\rm d}\sigma\le 
\frac{2}{\cos\theta}\,\|\varphi\|_{L^2(\Omega)}\|\nabla\varphi\|_{L^2(\Omega,{\mathbb{R}}^n)},
\end{equation}
where $\theta$ has been defined in \eqref{theta}. 
\end{theorem}

\begin{proof}
Let $\varphi:{\mathbb{R}}^n\to{\mathbb{R}}$ be a smooth function with compact support, 
and apply the divergence theorem in $\Omega$ with $u=\varphi^2\,e$ where 
$e=(0_{{\mathbb{R}}^{n-1}},1)$. Since 
${\rm div}\,(\varphi^2\,e)=2\,\varphi\,(e\cdot\nabla\varphi)$, we obtain
$$
\int_\Omega 2\,\varphi\,(e\cdot\nabla\varphi)\,{\rm d}x = 
\int_\Omega{\rm div}\,(\varphi^2\,e)\,{\rm d}x=
\int_\Gamma\nu\cdot(\varphi^2\,e)\,{\rm d}\sigma.
$$
Therefore using \eqref{costheta} and Cauchy-Schwarz inequality, we get
$$
\cos \theta \int_\Gamma \varphi^2\,{\rm d}\sigma
\le -2\,\int_\Omega \varphi\,(e\cdot\nabla\varphi)\,{\rm d}x
\le 2\,\|\varphi\|_{L^2(\Omega)}\|\nabla\varphi\|_{L^2(\Omega,{\mathbb{R}}^n)},
$$
which gives the estimate \eqref{traceH1}.
\end{proof}

\begin{corollary}
There exists a unique operator $T\in {\mathscr{L}}(H^1(\Omega),L^2(\Gamma))$ satisfying
$$
T\varphi={\rm Tr}_{|_{\Gamma}}\varphi,\quad \mbox{for all }
\varphi\in H^1(\Omega)\cap{\mathscr{C}}(\overline{\Omega})
$$
and 
\begin{equation}
\label{normT}
\|T\|_{{\mathscr{L}}(H^1(\Omega),L^2(\Gamma))}\le \frac{1}{\sqrt{\cos\theta}}.
\end{equation}
\end{corollary}

\begin{proof}
The existence and uniqueness of the operator $T$ follow from Theorem~\ref{H1} the density of 
${\mathscr{C}}_c^\infty(\overline{\Omega})$ in $H^1(\Omega)$ 
(see, e.g., \cite[Theorem~4.7, p.\,248]{EE87}). Moreover, \eqref{traceH1} implies
$$
\|\varphi\|_{L^2(\Gamma,{\rm d}\sigma)}^2\le 
\frac{1}{\cos\theta}\,\bigl(\|\varphi\|_{L^2(\Omega)}^2+
\|\nabla\varphi\|_{L^2(\Omega,{\mathbb{R}}^n)}^2\bigr),
\quad \mbox{for all }\varphi\in {\mathscr{C}}_c^\infty(\overline{\Omega}),
$$
which proves \eqref{normT}.
\end{proof}

{\small
\bibliographystyle{amsplain}

}
\end{document}